\documentclass[10pt,leqno]{amsart}
\usepackage{amsfonts}
\usepackage{amssymb}
\usepackage{amsmath}
\usepackage[all]{xy}
\pdfoutput=1

\let\oldsection=\section

\makeatletter
\renewcommand{\@seccntformat}[1]{\bf\@nameuse{the#1}.\quad}
\renewcommand\section{\@startsection{section}{1}%
                   \z@{.7\linespacing\@plus\linespacing}{.5\linespacing}%
                   {\normalfont\bfseries \boldmath}}
\renewcommand\subsection{\@startsection{subsection}{2}%
                   \z@{.5\linespacing\@plus.7\linespacing}{-.5em}%
                   {\normalfont\bfseries \boldmath}}
\renewcommand\subsubsection{\@startsection{subsubsection}{3}%
                   \z@{.3\linespacing\@plus.5\linespacing}{-.5em}%
                   {\normalfont\bfseries \boldmath}}
\makeatother

\newcommand{\CatO}{\mathcal{O}}
\newcommand{\La}{\mathfrak{g}}
\newcommand{\Lsa}{\mathfrak{t}}
\newcommand{\p}{\mathfrak{p}}
\newcommand{\levi}{\mathfrak{l}}

\newcommand{\nil}{\mathfrak{u}}
\newcommand{\wt}{{\sf wt}}

\newcommand{\Uq}{\mathcal{U}_q}

\newcommand{\lj}{\mathfrak{l}_J}
\newcommand{\pj}{\mathfrak{p}_J}
\newcommand{\ch}{{\sf ch}}

\theoremstyle{plain}
\newtheorem{thm}{Theorem}[subsection]

\newtheorem{lem}[thm]{Lemma}
\newtheorem{prop}[thm]{Proposition}

\theoremstyle{definition}
\newtheorem{defn}[thm]{Definition}

\numberwithin{equation}{subsection}

\newcounter{listequation}

\def\Z{{\mathbb Z}}

\def\C{\mathbb C}
\def\:{\colon}

\def\la{\lambda}

\def\Ker{\text{Ker }}

\def\Hom{\text{Hom}}
\def\Ext{\text{Ext}}

\def\WJ{{}^JW}

\def\ip#1,#2{\langle #1, #2 \rangle}
\def\ch{\operatorname{ch}}

\newcommand{\BU}{{\mathbb U}}

\def\rddots{\mathinner{\mkern1mu\raise1pt\vbox{\kern7pt\hbox{.}}\mkern2mu\raise4pt\hbox{.}\mkern2mu
\raise7pt\hbox{.}\mkern1mu}}

\def\Label#1{\label{#1}{\tt [#1]}}
\def\Label{\label}

\begin{document}


\title[Cohomology of quantum groups: An analog of Kostant's Theorem ]
{Cohomology of quantum groups: An analog of Kostant's Theorem}

\begin{abstract}
We prove the analog of Kostant's Theorem on Lie algebra cohomology in
the context of quantum groups. We prove that Kostant's cohomology
formula holds for quantum groups at a generic parameter $q$,
recovering an earlier result of Malikov in the case where the
underlying semisimple Lie algebra $\mathfrak{g} = \mathfrak{sl}(n)$.
We also show that Kostant's formula holds when $q$ is specialized to
an $\ell$-th root of unity for odd $\ell \ge h-1$ (where $h$ is the
Coxeter number of $\mathfrak{g}$) when the highest weight of the
coefficient module lies in the lowest alcove.  This can be regarded as
an extension of results of Friedlander-Parshall and Polo-Tilouine on
the cohomology of Lie algebras of reductive algebraic groups in prime
characteristic.  
\end{abstract}

\stepcounter{footnote}
\author[UNIVERSITY OF GEORGIA VIGRE ALGEBRA GROUP]{UNIVERSITY OF
GEORGIA VIGRE ALGEBRA GROUP${}^\thefootnote$}
\footnotetext{The members of the UGA VIGRE Algebra Group are
Irfan Bagci, Brian D. Boe, Leonard Chastkofsky, Benjamin Connell, Benjamin Jones,
Wenjing Li, Daniel K. Nakano, Kenyon J. Platt, Jae-Ho Shin, Caroline B. Wright.}

\address{Department of Mathematics \\
                 University of Georgia\\ Athens, Georgia 30602  }

\date{\today}



\dedicatory{}

\maketitle

\section{Introduction}

\subsection{} Let ${\mathfrak g}$ be a complex semisimple Lie algebra, $L(\lambda)$ be 
a finite dimensional ${\mathfrak g}$-module, ${\mathfrak p}_{J}$ be a parabolic 
subalgebra, ${\mathfrak u}_{J}$ be the nilradical of ${\mathfrak p}_{J}$ and ${\mathfrak l}_{J}$ 
be the associated Levi subalgebra. A celebrated result of Kostant states that the 
cohomology $\operatorname{H}^{\bullet}({\mathfrak u}_{J},L(\lambda))$ 
is the (multiplicity-free) direct sum of finite dimensional 
${\mathfrak l}_{J}$-modules, $L_{J}(w\cdot \lambda)$,  with $w\in {}^{J}W$ (minimal 
length coset representatives of $W_{J}\backslash W$ where $W_{J}$ is the Weyl group associated 
to the subset $J$ of simple roots). There are many proofs of this
result (e.g. \cite{Knapp,GW,RC}), including ones which 
involve constructing BGG resolutions. Recently the authors \cite{UGAVIGRE} 
provided a proof of Kostant's theorem by utilizing linkage in the parabolic category ${\mathcal O}_{J}$. 

In this paper we will prove an analog of Kostant's theorem in the case of quantum 
groups over generic parameters and at roots of unity. Our proof uses many of the ideas 
presented in \cite{UGAVIGRE}. More specifically, we first construct a subalgebra 
${\mathcal U}_q({\mathfrak u}_{J})$ of the quantized enveloping algebra ${\mathcal U}_{q}({\mathfrak g})$ 
by using the existence of a PBW basis given by Lusztig. We can then consider 
finite dimensional integrable modules for ${\mathcal U}({\mathfrak g})$, and calculate 
the cohomology $\operatorname{H}^{\bullet}({\mathcal U}_{q}({\mathfrak u}_{J}),L^{q}(\lambda))$ as 
a ${\mathcal U}_{q}({\mathfrak l}_{J})$-module. In the case when $q$ is a generic parameter 
our results recover an earlier result of Malikov \cite{M} when ${\mathfrak g}={\mathfrak sl}(n)$ and 
${\mathfrak u}_{J}={\mathfrak u}$ is the unipotent radical of a Borel subalgebra. We also consider the 
case when $q$ is specialized to an $l$th root of unity. In this situation, when $l\geq h-1$ 
($h$ being the Coxeter number of the underlying root system), we prove a version 
of Kostant's theorem for weights in the lowest alcove. These results can be regarded 
as an extension of results by Friedlander and Parshall \cite{FrPa:86b} and by Polo and Tilouine \cite[Corollary 2.10]{PT} for the cohomology of 
Lie algebras of reductive algebraic groups in prime characteristic. 
 
\subsection{} The paper is organized as follows. The basic information we shall need about quantum groups is in Section \ref{S:QuantumGroups}. In Section \ref{S:CategoryO} 
we introduce a parabolic 
version of Category ${\mathcal O}$ for quantum groups which generalizes the 
ordinary quantum group Category ${\mathcal O}^{q}$ (as defined in \cite{HK}). 
Within these categories blocks can be defined and the simple modules in a block are 
parametrized by standard ``linkage classes''. This information gives us 
upper bounds on the composition factors in the cohomology. In Section \ref{S:EulerCharacters}, 
we provide results which compare the Euler characteristics on the ordinary ${\mathfrak u}_{J}$-cohomology 
with the cohomology for ${\mathcal U}_{q}({\mathfrak u}_{J})$. These techniques provide 
effective lower bounds on the composition factors in $\operatorname{H}^{\bullet}({\mathcal U}_{q}({\mathfrak u}_{J}),L^{q}(\lambda))$.  With the results in Sections 
\ref{S:CategoryO} and \ref{S:EulerCharacters}, we 
prove the quantum version of Kostant's theorem in Section~\ref{S:Kostant} 
(Theorems~\ref{T:KostantK} and 
~\ref{T:KostantL}). In Section~\ref{S:RootsOfUnity}, we apply linkage in a 
graded version of the small 
quantum group $u_{\zeta}({\mathfrak g})$ where $\zeta$ is a primitive $l$th root of unity, along with 
the Euler characteristic results in Section \ref{S:EulerCharacters}, to prove the 
quantum version of the 
Polo-Tilouine result. 

\section{Quantum Groups} \Label{S:QuantumGroups}

\subsection{} We will follow the conventions as described in
\cite[Section 2]{BNPP}. Let ${\mathfrak g}$ be a complex simple Lie algebra. 
Let $\Phi$ be the irreducible root system associated to ${\mathfrak g}$
and $\Delta=\{\alpha_1, \ldots, \alpha_r\}$ be a fixed set of simple roots. The set $\Phi$ spans a real
vector space $\mathbb E$ with positive definite inner product
$\langle u,v\rangle$, $u,v\in \mathbb E$, adjusted so that
$\langle\alpha,\alpha\rangle=2$ if $\alpha\in\Phi$ is a short root.
For $\alpha\in\Phi$, let $\alpha^\vee=\frac{2}{\langle\alpha,\alpha\rangle}\alpha$.
For $J\subseteq \Delta$, let $\Phi_{J}=\Phi\cap {\mathbb Z}J$ be the root subsystem of
$\Phi$ generated by $J$. Let $W$ be the Weyl group corresponding to $\Phi$ and 
$W_{J}$ the Weyl group of $\Phi_{J}$, viewed as a subgroup of $W$. Let ${^J}W$ denote 
the set of minimal length coset representatives for $W_{J}\backslash W$.

Define the fundamental dominant weights $\varpi_1,\cdots,\varpi_r$
by $\langle\varpi_i,\alpha_j^\vee\rangle=\delta_{i,j}$, so
the (integral) weight lattice $X={\mathbb Z}\varpi_1\oplus \cdots\oplus{\mathbb Z}\varpi_r$,
and the set of dominant (integral) weights is
$X^+={\mathbb N}\varpi_1\oplus\cdots\oplus {\mathbb N}\varpi_r$, where ${\mathbb N}$
denotes the set of nonnegative integers. Let
$$X_J^+=\{\,\lambda\in X : \langle\lambda,\check\alpha\rangle \in\mathbb N \text{ for all } \alpha\in J\,\}$$ 
be the set of dominant integral weights for $J$. Set $\rho=\varpi_{1}+\dots+\varpi_{r}$.

Let ${\mathfrak t}$ be a fixed maximal toral subalgebra of $\mathfrak g$.
Given $\alpha\in\Phi$, let ${\mathfrak g}_\alpha$ be the $\alpha$-root space. Put ${\mathfrak
b}^+=\mathfrak t\oplus\bigoplus_{\alpha\in\Phi^+}{\mathfrak
g}_\alpha$ (the positive Borel subalgebra), and ${\mathfrak b}=
\mathfrak t\oplus\bigoplus_{\alpha\in\Phi^-}{\mathfrak g}_\alpha$
(the opposite Borel subalgebra). More generally, given a subset $J \subseteq \Delta$,
one can consider the Levi and parabolic Lie subalgebras ${\mathfrak l}_J$ and
${\mathfrak p}_J = {\mathfrak l}_J\oplus{\mathfrak u}_J$ of $\mathfrak g$.

\subsection{} Throughout this paper let $l>1$ be a fixed odd positive
integer. If $\Phi$ has type $G_2$, then we assume that $3$ does not divide
$l$. Let ${\mathcal A}={\mathbb Q}[q,q^{-1}]$ with fraction field ${\mathbb Q}(q)$.
Let $\zeta=\sqrt[l]{1}\in\mathbb C$ be a primitive $l$th root of unity
and $k={\mathbb Q}(\zeta)$. One can regard $k$ as an $\mathcal
A$-algebra by means of the homomorphism $\mathbb{Q}[q,q^{-1}]\to k$ where
$q\mapsto \zeta$.

The quantized enveloping algebra $\BU_q({\mathfrak g})$ of
$\mathfrak g$ is the ${\mathbb Q}(q)$-algebra with generators
$E_{\alpha}$, $K_{\alpha}^{\pm 1}$, $F_{\alpha}$, $\alpha\in \Delta$ and
relations (R1)--(R6) listed in \cite[(4.3)]{Jan3}.
The algebra $\BU_q({\mathfrak g})$ has two $\mathcal A$-forms,
$U_q^{\mathcal A}({\mathfrak g})$ (due to Lusztig) and
${\mathcal U}_q^{\mathcal A}({\mathfrak g})$ (due to De
Concini and Kac). Often we will use ${\mathcal U}_{q}({\mathfrak g})$ to denote 
the latter ${\mathcal A}$-form specialized to a generic parameter. After base change to $k$, these 
algebras play roles analogous to the hyperalgebra of a reductive group and the universal
enveloping algebra of its Lie algebra, respectively. Set
\begin{equation*}\label{firstquantumgroup}
U_\zeta({\mathfrak g}):=k\otimes_{\mathcal A}U^{\mathcal A}_q({\mathfrak g})/\langle1\otimes
K_\alpha^l-1\otimes 1,\alpha\in\Delta \rangle.\end{equation*}
The elements $E_\alpha,K_\alpha,F_\alpha$, $\alpha\in\Delta$,
in $U_{\zeta}({\mathfrak g})$ generate a
Hopf subalgebra, denoted $u_\zeta({\mathfrak g})$, of
$U_\zeta({\mathfrak g})$.

Recall that ${\mathcal U}^{\mathcal A}_q({\mathfrak g})$ is the $\mathcal
A$-subalgebra of $\BU_q({\mathfrak g})$ generated by the $E_\alpha, F_\alpha,
K_\alpha^{\pm 1}$, $\alpha\in\Delta$. There is an inclusion of ${\mathcal A}$-forms:
${\mathcal U}_q^{\mathcal A}({\mathfrak g})\subseteq {U}_q^{\mathcal A}({\mathfrak g})$.
Now set ${\mathcal U}_k({\mathfrak g}):=k\otimes_{\mathcal A}
{\mathcal U}^{\mathcal A}_q({\mathfrak g})$.
Finally, put
\begin{equation*}\label{secondquantumgroup}
{\mathcal U}_{\zeta}({\mathfrak g})={\mathcal
U}_k({\mathfrak g})/\langle 1\otimes K_\alpha^l-1\otimes 1,\alpha\in\Delta\rangle.
\end{equation*}
The (Hopf) algebra ${\mathcal U}_\zeta({\mathfrak g})$ has a central subalgebra
${\mathcal Z}$ such that $u_{\zeta}({\mathfrak g})\cong{\mathcal
U}_\zeta({\mathfrak g})//{\mathcal Z}$ (cf.\  \cite{DK} for more details).
The finite dimensional Hopf algebra $u_{\zeta}({\mathfrak g})$ will be referred to
as the {\em small quantum group}.

We will assume throughout that all $U_\zeta$-modules we consider are
integrable and type 1 (cf.\  \cite[Section 2.2]{BNPP}). Given such a module $V$, let $\wt(V)$ denote its 
set of weights.

\subsection{\bf Levi and Parabolic Subalgebra}\label{PBW}
For each $\alpha \in \Delta$, Lusztig has defined an automorphism
$T_{\alpha}$ of $\BU_q({\mathfrak g})$ (cf.\ \cite[Ch.\ 8]{Jan3}). If
$s_{\alpha}$ is a simple reflection in $W$, let $T_{s_{\alpha}} :=
T_{\alpha}$.  More generally, given any $w \in W$, let $w =
s_{\beta_1}s_{\beta_2}\cdots s_{\beta_n}$ be a reduced expression, and
define $T_w := T_{\beta_1}\cdots T_{\beta_n} \in
\text{Aut}(\BU_q({\mathfrak g}))$. The automorphism $T_w$ is
independent of the reduced expression of $w$.

Now let $J \subseteq \Delta$ and fix a reduced expression $w_0
=s_{\beta_1}\cdots s_{\beta_N}$ that starts with a reduced expression
for the long element $w_{0,J}$ for $W_{J}$.  The fixed reduced
expression for $w_{0}$ induces a fixed ordering on the positive roots,
which in turn leads to a definition of ``root vectors'' $E_{\gamma}$,
$F_{\gamma}$ for each ${\gamma}\in \Phi^{+}$, by using the
automorphisms above, as in \cite[Section 2.4]{BNPP}. Note that
$E_{\gamma}$ has weight $\gamma$, and $F_{\gamma}$ has weight
$-\gamma$. The vectors $E_{\gamma}$, $F_{\gamma}$, and $K^{\pm
  1}_{\alpha}$ for $\alpha\in\Delta$ form a PBW-like basis for
${\BU}_{q}(\mathfrak g)$.

The universal enveloping algebras of the Levi and parabolic
subalgebras associated to $J$ will be denoted by $\BU({\mathfrak
  l}_{J})$ and $\BU({\mathfrak p}_{J})$. One can naturally define
corresponding quantized enveloping algebras $\BU_q({\mathfrak l}_{J})$
and $\BU_q({\mathfrak p}_{J})$. As subalgebras of $\BU_q({\mathfrak
  g})$, $\BU_q({\mathfrak l}_{J})$ is generated by $\{E_{\alpha},
F_{\alpha} : \alpha \in J\} \cup \{K_{\alpha}^{\pm 1} : \alpha \in
\Delta\}$, and $\BU_q({\mathfrak p}_{J})$ is generated by
$\{E_{\alpha}: \alpha \in J\} \cup \{F_{\alpha}, K_{\alpha}^{\pm 1} :
\alpha \in \Delta\}$. Upon specialization one obtains the subalgebras
$U_{\zeta}({\mathfrak l}_{J})$, $U_{\zeta}({\mathfrak p}_{J})$,
$u_{\zeta}({\mathfrak l}_{J})$, $u_{\zeta}({\mathfrak p}_{J})$ of
$U_\zeta({\mathfrak g})$, and $\mathcal{U}_{\zeta}({\mathfrak l}_{J})$
and $\mathcal{U}_{\zeta}({\mathfrak p}_{J})$ of ${\mathcal
  U}_\zeta({\mathfrak g})$.  One can also make analogous constructions
with the opposite parabolic ${\mathfrak p}^{+}_{J}$.

With the PBW basis as described above one can define a subalgebra
$\BU_q({\mathfrak u}_{J})$ which is analogous to that of
$\BU({\mathfrak u}_{J}) \subset \BU({\mathfrak p}_{J})$.  Let $\Phi^+
\backslash \Phi_J^+ = \{ \gamma_{i_1}, \ldots, \gamma_{i_N} \}$ and
set $\BU_q({\mathfrak u}_{J})$ to be the subspace spanned by the
$F_{\gamma_{i_1}}^{a_{i_1}}\cdots F_{\gamma_{i_N}}^{a_{i_N}}$,
$a_{i_j} \in \mathbb{Z}_{\ge 0}$. According to \cite[Lemma
2.4.1]{BNPP} $\BU_q({\mathfrak u}_{J})$ is a subalgebra of
$\BU_q({\mathfrak p}_{J})$ and independent of the choice of reduced
expression for $w_0$. Again by specializing one obtains algebras
$U_{\zeta}({\mathfrak u}_{J})$ and $u_{\zeta}({\mathfrak u}_{J})$.

\section{Quantum Parabolic Category $\mathcal{O}$} \Label{S:CategoryO}

\subsection{Category ${\mathcal O}_{J}^{q}$}

Define a partial order on $X$ by
$$\mu\leq\lambda \quad \Longleftrightarrow \quad \lambda-\mu
\textrm{ is a sum of positive roots.}$$  Given $\lambda\in X$, define
$$D(\lambda)=\{\,\mu\in X: \mu\leq\lambda\,\}.$$  


The simple finite dimensional $\mathcal{U}_{q}(\lj)$-modules are
parameterized by the set $X_J^+$.  Denote the simple module of highest
weight $\lambda\in X_J^+$ by $L^{q}_J(\lambda)$.

\begin{defn}
The Category $\CatO_{J}^q$ 
consists of $\Uq(\La)$-modules $V$
 which satisfy the following conditions:
\begin{enumerate}
\item as a
$\mathcal{U}_q(\lj)$-module,
$$V=\bigoplus_{\lambda\in X_J^+}m_\lambda L^{q}_J(\lambda)$$ for some $m_\lambda\in
\mathbb{N}$;
\item there exist $\lambda_1,\lambda_2,\ldots ,\lambda_t\in X$
such that $$\wt(V)\subseteq D(\lambda_1)\cup
D(\lambda_2)\cup\cdots \cup D(\lambda_t).$$
\end{enumerate}
\end{defn}

\subsection{Verma Modules}

For each $\lambda \in X_J^+$, extend the finite dimensional
irreducible $\Uq(\lj)$-module 
$L^q_{J}(\lambda)$ to a
$\Uq(\pj)$-module by requiring that $F_\gamma
L_{J}(\lambda) = 0$ for all $\gamma \in \Phi^{+}
\smallsetminus \Phi_{J}^{+}$.

The \emph{parabolic Verma module} corresponding to $\lambda \in X_J^+$ is
defined by
$$Z_J^q(\lambda) := \Uq(\La) \otimes_{\Uq(\pj)}
L_J^q(\lambda) .$$
As in the classical setting (cf.\ \cite[Proposition 3.3]{RC}), the following properties hold:
\begin{enumerate}
\item $Z_J^q(\lambda)$ is an object of $\CatO_J^q$.
\item $Z_J^q(\lambda)$ is a highest weight module with highest weight $\lambda$.
\item $Z_J^q(\lambda)$ has a unique maximal submodule and its unique
  irreducible quotient is isomorphic to $L^q(\lambda)$, the
  irreducible $\Uq(\La)$ module with highest weight $\lambda$.
\item Every simple object in $\CatO^q_J$ is isomorphic to
  $L^q(\lambda)$ for some $\lambda \in X_J^+$.
\end{enumerate}

Specializing to $q=1$ (cf.\
\cite[Ch.\ 3]{HK}), we have an
identification of the character of $L^q(\lambda)$:
\begin{equation} \Label{E:SimpleCharactersAgree}
\ch L^q(\lambda) = \ch L(\lambda),
\end{equation}
where $L(\lambda)$ denotes the irreducible $\mathbb{U}(\La)$ module
with highest weight $\lambda$. (Note that this is not neccesarily
finite dimensional since we only require $\lambda \in X_J^+$.)

The Category $\CatO_{J}^q$ is a full subcategory of $\CatO^q$
(cf. \cite[\S 3.2]{HK}).  A crucial fact used in the sequel is that
the linkage principle holds in Category $\CatO^q_J$ since it holds in
$\CatO^q$ (cf. \cite[Claim 6.26]{jan:quantum}). In particular we use
the following fact. Consider a weight $\nu \in X_J^+$. If
$Z^q_J(\nu)$ has $L^q(\mu)$ as a composition factor then $\nu = w
\cdot \mu$ for some $w \in {^J}W$.

\subsection{Projective modules in $\CatO^q_J$}

We have a decomposition of Category $\CatO^q_J$ into infinitesimal
blocks $$\CatO_J^q=\bigoplus_{\mu \in\Lsa^*/W}\CatO_J^q(\mu ).$$
In this section, let $Z(\lambda) := Z_J^q(\lambda)$ for brevity.

\begin{prop}\label{thm:EPProp} The following hold:
\begin{enumerate}
\item[(a)]
    If $\lambda \in \Lsa^*$ is dominant, then $Z(\lambda)$ is
    projective.
\item[(b)]
    If $P\in \CatO_J^q$ is projective and $\dim L<\infty $,
    then $P\otimes L$ is projective in $\CatO_J^q$.
\item[(c)] Category $\CatO_J^q$ has enough projectives.
\end{enumerate}
\end{prop}

\textit{Proof.} The proof follows the line of reasoning given for ordinary Category $\CatO$ in \cite{Hum}. 
For completeness, we include the argument here.

(a) Start with an exact sequence
$M\longrightarrow N\longrightarrow 0$ in $\CatO_{J}^{q}$. Suppose $\varphi :Z(\lambda
)\longrightarrow N$ is a $\mathcal{U}_q(\La)$-module homomorphism.
We want a map $\widetilde\varphi:Z(\lambda )\longrightarrow M$ such
that
\begin{equation} \Label{E:ProjDiagram}
 \xymatrix{
& Z(\lambda) \ar@{-->}[dl]_{\tilde{\varphi }} \ar[d]^{\varphi } & \\
M \ar@{->>}[r]_{\pi } & N \ar[r] & 0 }
\end{equation}
commutes. Since $Z(\lambda )$ is a highest
weight module, it has a highest weight vector $v^+$, of weight
$\lambda $. Consider $\varphi (v^+)\in N$. The map
$\pi:M\longrightarrow N$ is surjective, so there exists $v \in M$
such that $\pi(v) = \varphi (v^+)$. Now, $Z(\lambda) \in
{\CatO_J^q}(\lambda)$. Furthermore,
$$N=\bigoplus_{\mu \in\Lsa^*/W}N^\mu$$ with $N^\mu \in \CatO_J^q(\mu)$,
and $\varphi (Z(\lambda ))\subseteq N^\lambda$. So without loss of
generality, we may suppose $M,N\in {\CatO_J^q}(\lambda) $. The
$\mathcal{U}_q (\mathfrak{n}^+)$-submodule of $M$ generated by $v$  has
a maximal weight vector of weight $\mu \geq \lambda $.  By the
Linkage Principle, $\mu =w \cdot \lambda$ for some $w \in W$. Thus
$\mu =\lambda$, since $\lambda $ is dominant. By the universal property
of $Z(\lambda )$, there exists $\widetilde\varphi:Z(\lambda
)\longrightarrow M$ such that
$\widetilde\varphi(v^+)=v$ and \eqref{E:ProjDiagram} commutes. Hence $Z(\lambda )$ is
projective.

(b) By the definition of Category $\CatO_J^q$, it is easy to see that
$V\otimes L\in \CatO_J^q$ for any $V\in\CatO_{J}^{q}$ and any finite
dimensional module $L$. In particular, $-\otimes L$ (and similarly
$L\otimes -$) are exact functors on $\CatO_{J}^{q}$. Since $P$ is
projective if and only if $\Hom_{\CatO_J^q}(P,-)$ is an exact functor,
$\Hom_{\CatO_J^q}(P,\Hom(L,-))\cong\Hom_{\CatO_{J}^{q}}(P,L^{*}\otimes
-)$ is the composition of two exact functors, hence exact.
Furthermore, $$\Hom_{\CatO_J^q}(P\otimes
L,-)\cong\Hom_{\CatO_J^q}(P,\Hom(L,-)).$$ Hence
$\Hom_{\CatO_J^q}(P\otimes L,-)$ is exact, therefore $P\otimes L$ is
projective in $\CatO_J^q$.

(c) We prove that if $M\in \CatO_J^q$, then there exists
a projective module $P\in \CatO_J^q$ such that $P\twoheadrightarrow
M$.  We proceed by induction on the length of $M$.
For the base step, for any $\lambda \in X_{J}^{+}$, we will find a
projective $P\in \CatO_J^q$ such that $P\twoheadrightarrow L^{q}(\lambda
)$. Let $\mu :=\lambda +n\rho$ for $n\geq 0$ large enough that $\mu$
is dominant; in particular, $\mu\in X_{J}^{+}$. Then the Verma module $Z(\mu)$ is projective by
part (a). Since $n\rho \in X^+$, $L^{q}(n\rho )$
is finite dimensional. Hence $P:=Z(\mu)\otimes L^{q}(n\rho )$ is
projective in $\CatO_{J}^{q}$ by part (b).

The standard argument using the tensor identity shows that if $\dim L<\infty $, 
then $Z(\xi )\otimes L$ has a
finite Verma filtration with subquotients of the form $Z(\xi +\nu
)$, where $\nu $ ranges over $\wt(L)$. Moreover, if $\nu $
is a minimal (respectively, maximal) weight of $L$, then $Z(\xi
+\nu)$ is a quotient (resp., submodule) of $Z(\xi)\otimes L$. The
lowest weight of $L^{q}(n\rho )$ is $-n\rho=w_0(n\rho)$, so
$Z(\mu-n\rho)=Z(\lambda)$ is a quotient of $P$. Consequently,
$$P\twoheadrightarrow Z(\lambda )\twoheadrightarrow L^{q}(\lambda )$$
since $L^{q}(\lambda)$ is a quotient of $Z(\lambda)$.

For the inductive step, suppose $l(M)> 1$ so that for some simple
module $L^{q}(\lambda)$ ($\lambda\in X_{J}^{+}$) we have a short exact sequence
$$0\longrightarrow L(\lambda)\longrightarrow M\longrightarrow N\longrightarrow 0.$$
Now, $l(N)<l(M)$ and so by our induction hypothesis, there exists a
projective module $Q\in\CatO_J^q$ such that we have a map
$\varphi:Q\twoheadrightarrow N$.  Consequently, there is a map
$\widetilde{\varphi}:Q\rightarrow M$ such that
$$ \xymatrix{
0 \ar[r] & L^{q}(\lambda) \ar[r] & M \ar[r]^{\pi } & N \ar[r] & 0 \\
& & & Q \ar@{->>}[ul]^{\tilde{\varphi }} \ar[u]_{\varphi } }
$$
commutes.  If $\widetilde{\varphi}$ is surjective, we are done.  If
$\widetilde\varphi $ is not surjective, then $M\cong N\oplus
L^{q}(\lambda )$.  Consequently, taking $P\in\CatO_J^q$ such that $P$ is
a projective cover of $L^{q}(\lambda)$, we have that $Q\oplus
P\twoheadrightarrow M$. $\Box$

\subsection{Relative Cohomology} In this subsection we outline basic definitions and results for relative cohomology .
 Let
$$ \cdots\rightarrow M_{i-1}\stackrel{f_{i-1}}{\longrightarrow} M_i\stackrel{f_i}{\rightarrow} M_{i+1}\rightarrow \cdots$$
be a sequence of $\mathcal{U}_q(\La)$-modules.  We say this sequence
is $(\mathcal{U}_q(\La),\mathcal{U}_q(\lj))$-\emph{exact} if it is exact as
a sequence of $\mathcal{U}_q(\La)$-modules and if, when viewed as a sequence of $\mathcal{U}_q(\lj)$-modules, Ker$f_i$ is a direct
summand of $M_i$ for all $i$.

A $\mathcal{U}_q(\La)$-module $P$ is $(\mathcal{U}_q(\La),\mathcal{U}_q(\lj))$-\emph{projective} if given any
$(\mathcal{U}_q(\La),\mathcal{U}_q(\lj))$-exact sequence
$$0\rightarrow M_1 \stackrel{f}{\rightarrow} M_2\stackrel{g}{\rightarrow} M_3\rightarrow 0,$$
and $\mathcal{U}_q(\La)$-module homomorphism $h:P\rightarrow M_3$ there is a $\mathcal{U}_q(\La)$-homomorphism $\tilde h : P\rightarrow M_2$ satisfying
$g\circ \tilde h=h$.

In particular, note that if $P$ is a projective $\mathcal{U}_q(\La)$-module, then it is automatically $(\mathcal{U}_q(\La),\mathcal{U}_q(\lj))$-projective.

A $(\mathcal{U}_q(\La),\mathcal{U}_q(\lj))$-\emph{projective resolution} of a $\mathcal{U}_q(\La)$-module is a
$(\mathcal{U}_q(\La),\mathcal{U}_q(\lj))$-exact sequence
$$\cdots\stackrel{\delta_2}{\rightarrow} P_1\stackrel{\delta_1}\rightarrow P_0\stackrel{\delta_0}{\rightarrow} M\rightarrow 0,$$
where each $P_i$ is a $(\mathcal{U}_q(\La),\mathcal{U}_q(\lj))$-projective module. We record a lemma whose 
proof is the same as in \cite{Kum}.

\begin{lem} The following statements hold:
\begin{enumerate}
\item[(a)]If $M$ is any $\mathcal{U}_q(\lj)$ module, then $\mathcal{U}_q(\La)\otimes _{\mathcal{U}_q(\lj)}M$ is
a $(\mathcal{U}_q(\La),\mathcal{U}_q(\lj))$-projective
 $\mathcal{U}_q(\La)$-module.
\item[(b)] Any $\mathcal{U}_q(\La)$-module admits a $(\mathcal{U}_q(\La),\mathcal{U}_q(\lj))$-projective resolution. Namely,
\begin{equation*}\cdots\stackrel{\delta_2}{\rightarrow} \mathcal{U}_q(\La)\otimes _{\mathcal{U}_q(\lj)} \Ker  \delta _0\stackrel{\delta_1}{\rightarrow}
\mathcal{U}_q(\La)\otimes _{\mathcal{U}_q(\lj)}M \stackrel
{\delta_0}{\rightarrow} M\rightarrow 0.
\end{equation*}Here $\delta_i$ is the ``multiplication'' map $\mathcal{U}_q(\La)\otimes _{\mathcal{U}_q(\lj)}N\rightarrow N$
given by $x\otimes n\mapsto xn$
for any $\mathcal{U}_q(\La)$-module $N$.
\end{enumerate}
\end{lem}

Let $\mathcal {C}=\mathcal{C}(\mathcal{U}_q(\La),\mathcal{U}_q(\lj))$ denote the 
full subcategory of all $\mathcal{U}_q(\La)$-modules which are
finitely semisimple as $\mathcal{U}_q(\lj)$-modules. Then the category $\mathcal{C}$ contains $\CatO_{J}^q$.
A key connection between these categories is that if $M$ and $N$ are objects of
$\CatO_{J}^q$, then one has
$$\Ext_{\CatO_{J}^q}^\bullet(M,N)\cong  \Ext_{(\mathcal{U}_q(\La),\mathcal{U}_q(\lj))}^\bullet(M,N) \cong\Ext_{\mathcal{C}}^\bullet(M,N). $$
(cf. \cite{Kum} and \cite{BNW}).

\subsection{Composition Factors}

If $V$ is a finitely semisimple $\Uq(\levi_J)$-module, write
$[V:L_{J}^q(\nu)]_{\Uq(\levi_J)}$ for the multiplicity of
$L_{J}^q(\nu)$ as a $\Uq(\levi_J)$-composition factor of $V$. The following 
theorems provide information about the composition factors inside the cohomology. 

\begin{thm} (cf.\  \cite[2.3.1]{UGAVIGRE}) \Label{T:UpperBound}
Let $V\in {\mathcal O}_{J}^q$ and let $\lambda\in X_{J}^{+}$.
\begin{itemize}
\item[(a)] $\operatorname{Ext}^{i}_{{\mathcal O}_{J}^q}
(Z_{J}^q(\lambda),V)\cong \operatorname{Hom}_{\Uq(\levi_{J})}
(L_{J}^q(\lambda),\operatorname{H}^{i}(\Uq(\mathfrak{u}_{J}),V))$
\item[(b)] If $[\operatorname{H}^{i}(\Uq(\mathfrak{u}_{J}),L^q(\mu))
:L_{J}^q(\nu)]_{\Uq(\levi_J)}\neq 0$ for $\mu\in X^{+}$, then
$\nu=w\cdot \mu$ for some $w\in {^J}W$.
\end{itemize}
\end{thm}

\begin{proof} (a) First observe that $\Ext^{i}_{{\mathcal O}_{J}^q}
(Z_{J}^q(\lambda),V)\cong \Ext^{i}_{(\Uq(\La),\Uq(\levi_{J}))}
(Z_{J}^q(\lambda),V)$ and by Frobenius reciprocity we have
\begin{equation*}
\Ext^{i}_{(\Uq(\La),\Uq(\levi_{J}))} (Z_{J}^q(\lambda),V)\cong
\Ext^{i}_{(\Uq(\p_{J}),\Uq(\levi_{J}))}(L_{J}^q(\lambda),V)\\
\cong
\operatorname{H}^{i}(\Uq(\p_{J}),\Uq(\levi_{J});L_{J}^q(\lambda)^{*}\otimes
V).
\end{equation*}

One can use the Grothendieck spectral sequence construction given in
\cite[I Proposition 4.1]{jan:rep} to obtain a spectral sequence,
{\setlength\arraycolsep{2pt}\begin{eqnarray*}
E_{2}^{i,j}&=&\operatorname{H}^{i}(\Uq(\p_{J})//\Uq(\nil_{J}),
\Uq(\levi_{J})/\Uq(\levi_{J} \cap \nil_{J});
\operatorname{H}^{j}(\Uq(\nil_{J}),0;L_{J}^q(\lambda)^{*}\otimes
V)) \\  &\Rightarrow&
\operatorname{H}^{i+j}(\Uq(\p_{J}),\Uq(\levi_{J});L_{J}^q(\lambda)^{*}\otimes
V)
\end{eqnarray*}}
(see \cite[Sec.\ 5.2]{GK} for the definition of $//$).
However,  $E_{2}^{i,j}\cong
\operatorname{H}^{i}(\Uq(\levi_{J}),\Uq(\levi_{J});
\operatorname{H}^{j}(\Uq(\levi_{J}),0;
L_{J}^q(\lambda)^{*}\otimes V))=0$ for $i>0$, so the spectral
sequence collapses and yields
{\setlength\arraycolsep{2pt}\begin{eqnarray*}
\text{Hom}_{\Uq(\levi_{J})}(L_{J}^q(\lambda),\operatorname{H}^{j}(\Uq(\nil_{J}),V))
&\cong&
\text{H}^{0}(\Uq(\levi_{J}),\Uq(\levi_{J});\operatorname{H}^{j}(\Uq(\nil_{J}),L_{J}^q(\lambda)^{*}\otimes
V) \\ &\cong&
\operatorname{H}^{j}(\Uq(\p_{J}),\Uq(\levi_{J});L_{J}^q(\lambda)^{*}\otimes
V).
\end{eqnarray*}}  

(b) Suppose that
$[\operatorname{H}^{i}(\Uq(\mathfrak{u}_{J}),L^q(\mu))
:L_{J}^q(\nu)]_{\Uq(\levi_J)}\neq 0$. Then from part (a),
{\setlength\arraycolsep{2pt}\begin{eqnarray*}
[\operatorname{H}^{i}(\Uq(\mathfrak{u}_{J}),L^q(\mu))
:L_{J}^q(\nu)]_{\Uq(\levi_J)} &=& \dim \Hom_{\Uq(\levi_{J})}
(L_{J}^q(\nu),\operatorname{H}^{i}(\Uq(\mathfrak{u}_{J}),L^q(\mu))
\\ &=& \dim \Ext^{i}_{{\mathcal O}_{J}^q}
(Z_{J}^q(\nu),L^q(\mu)).
\end{eqnarray*}}But, $\Ext^{i}_{{\mathcal O}_{J}^q}
(Z_{J}^q(\nu),L^q(\mu))\neq 0$ implies by linkage that $\nu = w\cdot
\mu$ for some $w\in {^J}W$.
\end{proof}

\begin{thm} (cf.\  \cite[2.5.1]{UGAVIGRE}) \Label{T:RelatingCFs}
Let $V$ be a finite dimensional $\Uq(\p_J)$-module. If
$$[ \operatorname{H}^i( \Uq(\nil_J), V ) \colon L^q(\lambda) ]_{\Uq(\lj)} \ne 0$$
then,
$$[ \operatorname{H}^i( \Uq(\nil_J), \C ) \otimes V \colon L^q(\lambda)
]_{\Uq(\lj)} \ne 0 .$$
\end{thm}

\begin{proof}
  Consider the case when $V$ is simple. Let $\Uq(\nil_J)_+$ denote
  the augmentation ideal of $\Uq(\nil_J)$ as in \cite[2.7]{BNPP}. Then
  $\Uq(\nil_J)_+$ acts by zero on $V$. This is a consequence of
  two properties. First, $\Uq(\nil_J)$ is a normal subalgebra of
  $\Uq(\pj)$ (cf.\  \cite[Corollary 2.7A]{BNPP}). Second, $\Uq(\nil_J)$
  has only one irreducible representation, the trivial module.  Taking
  $\Uq(\nil_J)_+$ invariants, we have a non-zero (because of the
  second property) subspace of $V$ which is invariant under
  $\Uq(\pj)$ (by the first property). Thus $\Uq(\nil_J)_+$ must act by
  zero on all of $V$.

  The trivial action on $V$ implies that $$\operatorname{H}^i(\Uq(\nil_J), V)
  \cong \operatorname{H}^i(\Uq(\nil_J), \C) \otimes V .$$ Thus the statement holds
  in this case. The general argument follows by induction on the
  length of a composition series for $V$ and the long exact sequence
  in cohomology as in \cite[Proposition 2.5.1]{UGAVIGRE}.
\end{proof}

\section{Euler Characters} \Label{S:EulerCharacters}

\subsection{} Let $M$ be a module in $\CatO^q_J$. From \cite[Proposition 2.9.1(b)]{BNPP} and
Theorem \ref{T:RelatingCFs}, 
$$\text{H}^{n}({\mathcal U}_{q}({\mathfrak u}_{J}),M)=0$$ for $n>\dim {\mathfrak u}_{J}$. 
Set
$$\chi(M)=\sum_{n=0}^{\dim {\mathfrak u}_{J}} (-1)^{n}\, \text{ch }\text{H}^{n}({\mathcal U}_{q}({\mathfrak u}_{J}),M).$$
If $M$ has a trivial ${\mathcal U}_{q}({\mathfrak u}_{J})$-action then
\begin{equation} \Label{E:ChiTrivialAction}
\chi(M)=\sum_{n=0}^{\dim {\mathfrak u}_{J}} (-1)^{n}\, \text{ch }(\text{H}^{n}({\mathcal U}_{q}({\mathfrak u}_{J}),{\mathbb C})\otimes M).
\end{equation}
Now observe that if $0\rightarrow M_{1} \rightarrow M \rightarrow M_{2} \rightarrow 0$ is
a short exact sequence of ${\mathcal U}_{q}({\mathfrak p}_{J})$-modules where
$M_{1}$ and $M_{2}$ have trivial ${\mathcal U}_{q}({\mathfrak u}_{J})$-action then
\begin{eqnarray*}
\chi(M)&=&\chi(M_{1})+\chi(M_{2})\\
       &=&\sum_{n=0}^{\dim {\mathfrak u}_{J}} (-1)^{n}\, \text{ch }(\text{H}^{n}({\mathcal U}_{q}({\mathfrak u}_{J}),{\mathbb C})\otimes M_{1})
+\sum_{n=0}^{\dim {\mathfrak u}_{J}} (-1)^{n}\, \text{ch }(\text{H}^{n}({\mathcal U}_{q}({\mathfrak u}_{J}),{\mathbb C})\otimes M_{2})\\
&=&\sum_{n=0}^{\dim {\mathfrak u}_{J}} (-1)^{n}\, \text{ch }(\text{H}^{n}({\mathcal U}_{q}({\mathfrak u}_{J}),{\mathbb C})\otimes M).
\end{eqnarray*}
By inducting on the composition length of $M$ and using the fact that all simple
${\mathcal U}_{q}({\mathfrak p}_{J})$-modules are obtained by inflating simple
${\mathcal U}_{q}({\mathfrak l}_{J})$-modules (by letting ${\mathcal U}_{q}({\mathfrak u}_{J})$ act
trivially), the formula \eqref{E:ChiTrivialAction} holds for all $M$. In particular,
\begin{equation}\Label{euler-irreducible}
\chi(L^{q}(\lambda))=\sum_{n=0}^{\dim {\mathfrak u}_{J}} (-1)^{n} \text{ch }\left(\text{H}^{n}({\mathcal U}_{q}({\mathfrak u}_{J}),{\mathbb C})
\otimes L^{q}(\lambda)\right).
\end{equation}

\subsection{} \Label{SS:EulerComparison} We will now compare Euler characterstics in the quantum and classical
cases. First observe that from \eqref{euler-irreducible} and \cite[Proposition 2.9.1(a)]{BNPP},
\begin{equation} \Label{E:chiLq}
\chi(L^{q}(\lambda))=\sum_{n=0}^{\dim {\mathfrak u}_{J}} (-1)^{n} [\text{ch }\text{H}^{n}({\mathcal U}_{q}({\mathfrak u}_{J}),{\mathbb C})]
[\text{ch }L^{q}(\lambda)]= \sum_{n=0}^{\dim {\mathfrak u}_{J}} (-1)^{n} [\text{ch }\Lambda_{q,J}^{n}][\text{ch }L^{q}(\lambda)].
\end{equation} See \cite[\S 2.9]{BNPP} for the definition of the quantum exterior algebra $\Lambda^{\bullet}_{q,J}$.

On the other hand, the analogous formula holds for the classical case:
\begin{equation*}
\chi(L(\lambda))=\sum_{n=0}^{\dim {\mathfrak u}_{J}} (-1)^{n} [\text{ch }\text{H}^{n}({\mathfrak u}_{J},{\mathbb C})]
[\text{ch }L(\lambda)]
= \sum_{n=0}^{\dim {\mathfrak u}_{J}} (-1)^{n} [\text{ch }\Lambda^{n}({\mathfrak u}_{J}^{*})][\text{ch }L(\lambda)].
\end{equation*}
It follows immediately from the definitions that for each $n$, 
\begin{equation} \Label{E:ExteriorCharacters}
\text{ch }\Lambda_{q,J}^{n} = \text{ch }
\Lambda^{n}({\mathfrak u}_{J}^{*}).
\end{equation}
Recalling \eqref{E:SimpleCharactersAgree}, we have $[\text{ch }\Lambda_{q,J}^{n}][\text{ch }L^{q}(\lambda)]=[\text{ch }
\Lambda^{n}({\mathfrak u}_{J}^{*})][\text{ch }L(\lambda)]$. Therefore,
\begin{equation}\label{comparison}
\sum_{n=0}^{\dim {\mathfrak u}_{J}} (-1)^{n} \text{ch }\text{H}^{n}({\mathcal U}_{q}({\mathfrak u}_{J}),L^{q}(\lambda))=
\sum_{n=0}^{\dim {\mathfrak u}_{J}} (-1)^{n} \text{ch }\text{H}^{n}({\mathfrak u}_{J},L(\lambda)).
\end{equation}

Our strategy will be to use the fact that the right hand side of \eqref{comparison} is given by the classical Kostant's Theorem.

\section{Kostant's Theorem} \Label{S:Kostant}

\subsection{} In this section we will prove an analog of Kostant's theorem for
quantum groups. We begin by proving the result for trivial coefficients.

\begin{thm} \Label{T:KostantK}
Let $J\subseteq \Delta$ and $q$ be a generic parameter. Then as a ${\mathcal U}_{q}({\mathfrak l}_{J})$-module,
$$\operatorname{H}^{n}({\mathcal U}_{q}({\mathfrak u}_{J}),{\mathbb C})
\cong \bigoplus_{\substack{w\in \WJ\\ l(w)=n}} L_{J}^{q}(w\cdot 0).$$
\end{thm}

\begin{proof} According to Theorem \ref{T:UpperBound}, all ${\mathcal U}_{q}({\mathfrak l}_{J})$-composition
factors of $\operatorname{H}^{n}({\mathcal U}_{q}({\mathfrak u}_{J}),{\mathbb C})$ are
of the form $L_{J}^{q}(w\cdot 0)$ where $w\in \WJ$. Furthermore, by
\cite[Proposition 2.9.1]{BNPP}, if $w\cdot 0$ is a weight of
$\operatorname{H}^{n}({\mathcal U}_{q}({\mathfrak u}_{J}),{\mathbb C})$ then
$w\cdot 0$ is a weight of $\Lambda^{n}_{q,J}$.
Recall (e.g., from \eqref{E:ExteriorCharacters} and \cite[Lemma 3.1.2]{UGAVIGRE}) that $\dim (\Lambda^{n}_{q,J})_{w\cdot 0}=0$ for $l(w)\neq n$ and
$(\Lambda^{n}_{q,J})_{w\cdot 0}\cong k$ for $l(w)=n$. 

Now by the classical Kostant's Theorem (cf.\  \cite[Theorem 4.2.1]{UGAVIGRE}), 
$\text{ch L}_{q}(w\cdot 0)$
appears exactly once in $\sum_{n=0}^{\dim {\mathfrak u}_{J}} (-1)^{n} \text{ch }\Lambda_{q,J}^{n}$;
thus, it must appear exactly once in $\sum_{n=0}^{\dim {\mathfrak u}_{J}} (-1)^{n} \text{ch }\text{H}^{n}({\mathcal U}_{q}({\mathfrak u}_{J}),{\mathbb C})$.
But by the previous paragraph, $[\operatorname{H}^{n}({\mathcal U}_{q}({\mathfrak u}_{J}),{\mathbb C}):L_{J}^{q}(w\cdot 0)]=0$
for $l(w)\neq n$, thus $\text{ch }L^{q}_{J}(w\cdot 0)$ appears exactly once in
$\text{ch }\operatorname{H}^{l(w)}({\mathcal U}_{q}({\mathfrak u}_{J}),{\mathbb C})$ and
$[\operatorname{H}^{l(w)}({\mathcal U}_{q}({\mathfrak u}_{J}),{\mathbb C}):L_{J}^{q}(w\cdot 0)]=1$.
The result now follows because of the semisimplicity of the cohomology as
${\mathcal U}_{q}({\mathfrak l}_{J})$-module.
\end{proof}

\subsection{} We can now apply the preceding theorem to compute the
cohomology of ${\mathcal U}_{q}({\mathfrak u}_{J})$ with coefficients
in a finite dimensional simple ${\mathcal U}_{q}({\mathfrak g})$-module.

\begin{thm} \Label{T:KostantL}
Let $J\subseteq \Delta$, $\mu\in X^{+}$ and $q$ a generic parameter.
Then as a ${\mathcal U}_{q}({\mathfrak l}_{J})$-module,
$$\operatorname{H}^{n}({\mathcal U}_{q}({\mathfrak u}_{J}),L^{q}(\mu))
\cong \bigoplus_{\substack{w\in \WJ\\ l(w)=n}} L^{q}_{J}(w\cdot \mu).$$
\end{thm}

\begin{proof} First note that by Theorem \ref{T:UpperBound}, all ${\mathcal U}_{q}({\mathfrak l}_{J})$-composition
factors of $\operatorname{H}^{n}({\mathcal U}_{q}({\mathfrak u}_{J}),L^{q}(\mu))$ are
of the form $L_{J}^{q}(w\cdot \mu)$ where $w\in \WJ$. Next observe
that if $[\operatorname{H}^{n}({\mathcal U}_{q}({\mathfrak u}_{J}),L^{q}(\mu)):L^{q}_{J}(w\cdot \mu)]\neq 0$
then $[\operatorname{H}^{n}({\mathcal U}_{q}({\mathfrak u}_{J}),{\mathbb C})\otimes L^{q}(\mu) :L^{q}_{J}
(w\cdot \mu)]\neq 0$. In other words, 
$$\text{Hom}_{U_{q}({\mathfrak l}_{J})}(L^{q}_{J}(w\cdot \mu),
\operatorname{H}^{n}({\mathcal U}_{q}({\mathfrak u}_{J}),{\mathbb C})\otimes L^{q}(\mu))\neq 0.$$ 

We are now in a position to use the argument in \cite[Theorem 4.2.1]{UGAVIGRE}.
We first apply Frobenius reciprocity and Theorem~\ref{T:KostantK} to deduce that
$$w\cdot \mu=w^{\prime}\cdot 0+\nu,$$ where $\nu\in \wt(L^{q}(\mu))$ and $l(w^{\prime})=n$.
The argument in \cite{UGAVIGRE} shows that $w=w^{\prime}$, thus $l(w)=n$. Furthermore,
$$\operatorname{H}^{n}({\mathcal U}_{q}({\mathfrak u}_{J}),{\mathbb C})_{w\cdot 0}=1.$$
We can now conclude that
\begin{equation}
[\operatorname{H}^{n}({\mathcal U}_{q}({\mathfrak u}_{J}),L^{q}(\mu)):L^{q}_{J}(w\cdot \mu)]\leq 1
\end{equation}
for all $w\in \WJ$.

The cohomology is completely reducible as ${\mathcal U}_{q}({\mathfrak l}_{J})$-module so it
remains to prove that
\begin{equation}
[\operatorname{H}^{n}({\mathcal U}_{q}({\mathfrak u}_{J}),L^{q}(\mu)):L^{q}_{J}(w\cdot \mu)]=1
\end{equation}
for all $w\in \WJ$ and $l(w)=n$. Suppose that $w\in \WJ$ and
$l(w)=n$. By Kostant's Theorem in the classical case, $\text{ch
}L(w\cdot \mu)$ occurs exactly once in $\sum_{i=0}^{\dim {\mathfrak
    u}_{J}} (-1)^{i} \text{ch }\text{H}^{i}({\mathfrak
  u}_{J},L(\mu))$. From (\ref{comparison}), $\text{ch
}L^{q}_{J}(w\cdot \mu)$ appears exactly once in $\sum_{i=0}^{\dim
  {\mathfrak u}_{J}} (-1)^{i} \text{ch }\text{H}^{i}({\mathcal
  U}_{q}({\mathfrak u}_{J}), L^{q}(\lambda))$. From the preceding
paragraph $L^{q}_{J}(w\cdot \mu)$ can only appear as a composition
factor, with multiplicity at most one, in $\text{H}^{i}({\mathcal U}_{q}({\mathfrak
  u}_{J}),L^{q}(\lambda))$ where $i=n=l(w)$. Putting these statements
together implies that $[\operatorname{H}^{n}({\mathcal
  U}_{q}({\mathfrak u}_{J}),L^{q}(\mu)):L^{q}_{J}(w\cdot \mu)]=1$.
\end{proof}

\section{Roots of Unity} \Label{S:RootsOfUnity}

\subsection{} In this section we will demonstrate that Kostant's theorem has a quantum
analog in the case when the parameter is a primitive $l$th root of unity under the condition
that $l\geq h-1$ and the highest weight of the simple ${\mathcal U}_{\zeta}({\mathfrak g})$-module
is in the bottom alcove. These results were inspired by Polo and Tilouine's result for reductive
groups in positive characteristic (cf.\  \cite[Theorem 2.1]{PT}, \cite[Theorem 4.2.1]{UGAVIGRE}).

\subsection{} The first step is to prove an upper bound on the composition factors of the
cohomology groups $\operatorname{H}^{\bullet}({\mathcal U}_{\zeta}({\mathfrak u}_{J}),L^{\zeta}(\mu))$.
In order to do so we need some additional notation. Since $u_{\zeta}({\mathfrak g})$ is a normal
sub Hopf algebra of $U_{\zeta}({\mathfrak g})$, one can form the algebra $Ku_{\zeta}({\mathfrak g})$ where
$K$ is the subalgebra of $U_{\zeta}({\mathfrak g})$ generated by $\{K_{\alpha}^{\pm 1}: \alpha \in \Delta\}$. The same
construction works if one replaces ${\mathfrak g}$ by ${\mathfrak b}$, ${\mathfrak p}_{J}$ or ${\mathfrak l}_{J}$.
This construction is dual to the construction given in \cite[Chapter 9]{PW} and analogous to the graded $G_{1}T$-category
for reductive algebraic group schemes $G$ in positive characteristic. For a fixed $l>1$, let
$W_{l}=W\ltimes l{\mathbb Z}\Phi$ be the affine Weyl group and $\widehat W_{l}=W\ltimes lX$ be the extended
affine Weyl group. Write $X_{1}$ for the set of $l$-restricted weights
$$
X_{1}=\{\, \la\in X : 0 \le \langle \la, \check\alpha \rangle < l \text{ for all } \alpha\in\Delta \,\},
$$
and similarly $(X_{J})_{1}$ for the set of $l$-restricted $J$-weights.

The following theorem provides information about the ${\mathcal U}_{\zeta}({\mathfrak l}_{J})$-composition factors in the
${\mathcal U}_{\zeta}({\mathfrak u}_{J})$-cohomology. 

\begin{thm} \Label{T:Ccompfacts} Let $J\subseteq \Delta$ and $\zeta$ be a primitive 
$l$th root of unity.
\begin{itemize}
\item[(a)] If $[\operatorname{H}^{i}({\mathcal U}_{\zeta}({\mathfrak u}_{J}),L^{\zeta}(\mu))
:L^{\zeta}_{J}(\sigma)]_{{\mathcal U}_{\zeta}({\mathfrak l}_{J})}\neq 0$ where $\mu\in X^{+}$ then $\mu=w\cdot
\sigma$ where $w\in \widehat{W}_{l}$.
\item[(b)] If $[\operatorname{H}^{i}({\mathcal U}_{\zeta}({\mathfrak
u}_{J}),L^{\zeta}(\mu)):L^{\zeta}_{J}(\sigma)]_{{\mathcal U}_{\zeta}(l_{J})}\neq 0$ where $\mu\in X_{1}$ and
$\sigma\in (X_{J})_{1}$ then $\mu=w\cdot \sigma$ where $w\in {W}_{l}$.
\end{itemize}
\end{thm}

\begin{proof} We will provide a sketch of the argument which is similar to the one given
in \cite[Theorem 2.4.1]{UGAVIGRE}. It should be noted that one needs to deviate somewhat
from the original proof for Frobenius kernels because a different spectral sequence construction is
needed in the quantum case.

(a) The first part of the proof follows by using the argument provided in \cite[Theorem 2.4.1]{UGAVIGRE}
by replacing $L_{J}$ with ${\mathcal U}_{\zeta}({\mathfrak l}_{J})$ and
$(L_{J})_{1}T$ with the algebra $Ku_{\zeta}({\mathfrak l}_{J})$, etc. The main points are that
if $[\operatorname{H}^{i}({\mathcal U}_{\zeta}({\mathfrak u}_{J}),L^{\zeta}(\mu))
:L^{\zeta}_{J}(\sigma)]_{{\mathcal U}_{\zeta}({\mathfrak l}_{J})}\neq 0$ where $\mu\in X^{+}$ then
one can express $\mu=\mu_{0}+l\mu_{1}$ where $\mu_{0}\in X_{1}$, $\mu_{1}\in X^{+}$, and 
$\sigma=\sigma_{0}+l\sigma_{1}$ where $\sigma_{0}\in (X_{J})_{1}$, $\sigma_{1}\in X_{J}^{+}$.
Then one shows that
$$[\operatorname{H}^{i}({\mathcal U}_{\zeta}({\mathfrak u}_{J}),L^{\zeta}(\mu_{0})):
L_{J}^{\zeta}(\sigma_{0})\otimes l\gamma]_{Ku_{\zeta}({\mathfrak l}_{J})}\neq 0$$ for
some $\gamma\in X$. Here $\gamma$ is the difference between a weight in the simple ${\mathcal U}_{\zeta}(\mathfrak g)$-module $L^{\zeta}(\mu_{1})$ and 
a weight in the simple ${\mathcal U}_{\zeta}(\mathfrak l_{J})$-module $L^{\zeta}_{J}(\sigma_{1})$.  

Here is where the proof now departs from the one given in \cite[Theorem 2.4.1(a)]{UGAVIGRE}. Let
${\mathcal Z}$ be the subalgebra in ${\mathcal U}_{\zeta}({\mathfrak g})$ generated by
$\{E_{\gamma}^{l},F_{\gamma}^{l}: \gamma\in \Phi\}$ and set
${\mathcal Z}_{J}={\mathcal U}_{\zeta}({\mathfrak u}_{J})\cap {\mathcal Z}$.
Then ${\mathcal U}_{\zeta}({\mathfrak u}_{J})//{\mathcal Z}_{J}\cong u_{\zeta}({\mathfrak u}_{J})$. From
\cite[Section 5.4]{BNPP} there exists a spectral sequence:
\begin{equation*}
{E}_{2}^{a,b}=\text{H}^{a}(u_{\zeta}({\mathfrak u}_{J}),\text{H}^{b}({\mathcal Z}_{J},
L^{\zeta}(\mu_{0})))\Rightarrow \text{H}^{a+b}({\mathcal U}_{\zeta}({\mathfrak u}_{J}),L^{\zeta}(\mu_{0})).
\end{equation*}
Since the algebra ${\mathcal Z}_{J}$ is central and generated by nilpotent elements, the action of
${\mathcal Z}_{J}$ on $L^{\zeta}(\mu)$ is trivial. Moreover, the action of $u_{\zeta}({\mathfrak u}_{J})$ on
$\text{H}^{\bullet}({\mathcal Z}_{J},{\mathbb C})$ is trivial. Therefore, one can rewrite the
spectral sequence as
\begin{equation*}
{E}_{2}^{a,b}=\text{H}^{a}(u_{\zeta}({\mathfrak u}_{J}),L^{\zeta}(\mu_{0}))\otimes
\text{H}^{b}({\mathcal Z}_{J},{\mathbb C})\Rightarrow \text{H}^{a+b}({\mathcal U}_{\zeta}({\mathfrak u}_{J}),L^{\zeta}(\mu_{0}))).
\end{equation*}
Now suppose that $[\operatorname{H}^{i}({\mathcal U}_{\zeta}({\mathfrak u}_{J}),L^{\zeta}(\mu_{0})):
L_{J}^{\zeta}(\sigma_{0})\otimes l\gamma]_{Ku_{\zeta}({\mathfrak l}_{J})}\neq 0$. Then
$$\text{Hom}_{Ku_{\zeta}({\mathfrak l}_{J})}(P,\operatorname{H}^{i}({\mathcal U}_{\zeta}({\mathfrak u}_{J}),
L^{\zeta}(\mu_{0})))\neq 0$$ where $P:=P_{J}^{\zeta}(\sigma_{0})\otimes l\gamma$ is the 
$Ku_{\zeta}({\mathfrak l}_{J})$-projective cover of $L_{J}^{\zeta}(\sigma_{0})\otimes l\gamma$. From 
the spectral sequence one can deduce that
$$\text{Hom}_{Ku_{\zeta}({\mathfrak l}_{J})}(P,\text{H}^{a}(u_{\zeta}({\mathfrak u}_{J}),L^{\zeta}(\mu_{0}))\otimes
\text{H}^{b}({\mathcal Z}_{J},{\mathbb C}))\neq 0$$
where $i=a+b$.

The weights of $\text{H}^{\bullet}({\mathcal Z}_{J},{\mathbb C})$ are of the form $l\delta$ where
$\delta\in {\mathbb Z}\Phi$. This implies that
\begin{eqnarray*}
0 & \neq & \dim \text{Hom}_{Ku_{\zeta}({\mathfrak l}_{J})}(P,\text{H}^{a}(u_{\zeta}({\mathfrak u}_{J}),L^{\zeta}(\mu_{0}))\otimes l\delta)\\
& = & [\operatorname{H}^{a}(u_{\zeta}({\mathfrak u}_{J}),L^{\zeta}(\mu_{0})):
L_{J}^{\zeta}(\sigma_{0})\otimes l\gamma\otimes (-l\delta)]_{Ku_{\zeta}({\mathfrak l}_{J})}\\
&=& \dim \text{Ext}^{a}_{Ku_{\zeta}({\mathfrak g})}(\text{coind}_{Ku_{\zeta}({\mathfrak p}_{J})}^{Ku_{\zeta}({\mathfrak g})} P\otimes (-l\delta), L^{\zeta}(\mu_{0})),
\end{eqnarray*}
where the last equality is deduced via an argument using the Lyndon-Hochschild-Serre spectral sequence.

Therefore, the linkage principle for blocks in $Ku_{\zeta}({\mathfrak g})$ implies that
$\sigma_{0}+l\gamma-l\delta$ and $\mu_{0}$ are linked under ${W}_{l}$, thus $\sigma_{0}$ and $\mu_{0}$ are 
linked under $\widehat{W}_{l}$.

(b) Under the hypotheses, we have $\gamma=0$ as in the proof of \cite[Theorem 2.4.1]{UGAVIGRE}. In this
case $\mu=\mu_{0}$ and $\sigma=\sigma_{0}$, and since $\delta\in\Z\Phi$, we deduce that $\mu=w\cdot \sigma$ where $w\in {W}_{l}$.
\end{proof}

\subsection{} One can verify that the proofs of \cite[Proposition 3.5.1, 3.6.1]{UGAVIGRE}
depend only on using weight estimates which apply to our setting.
A conversion of these results to the quantum setting yields the following results.

\begin{prop} \Label{P:weightsbottom} Let $\zeta$ be a primitive $l$th root of unity with $l\ge h-1$.
\begin{itemize}
\item[(a)] Suppose $\sigma=w\cdot 0 + l\mu$ is a weight of
$\Lambda^\bullet_{\zeta,J}$ where $w\in W$ and $\mu\in X$. Then
$\sigma=x\cdot 0$ for some $x\in W$.
\item[(b)] If $w\in \WJ$ then $L^{\zeta}_{J}(w\cdot 0)$ is in the bottom alcove for ${\mathcal U}_{\zeta}(
{\mathfrak l}_{J})$.
\item[(c)] If $w\in \WJ$ and $\lambda\in \overline{C}_{\mathbb Z}\cap X^{+}$ then
$L^{\zeta}_{J}(w\cdot 0)\otimes L^{\zeta}(\lambda)$ is completely
reducible as a ${\mathcal U}_{\zeta}({\mathfrak l}_{J})$-module.
\end{itemize}
\end{prop}

We also need to revisit Section \ref{S:EulerCharacters} on Euler characteristics. 
The formula \eqref{E:chiLq} and
\cite[Lemma 2.9.1(a)]{BNPP} hold when $q$ is a root of unity. Therefore,
$$\chi(L^{\zeta}(\lambda))=\sum_{n=0}^{\dim {\mathfrak u}_{J}}(-1)^{n} [\text{ch }\Lambda_{\zeta,J}^{n}]
[\text{ch }L^{\zeta}(\lambda)].$$
Observe that when $\lambda\in \overline{C}_{\mathbb Z}\cap X^{+}$ then for each $n$
\begin{equation}
[\text{ch }\Lambda_{q,J}^{n}][\text{ch }L_{\zeta}(\lambda)]=[\text{ch }
\Lambda^{n}({\mathfrak u}_{J}^{*})][\text{ch }L(\lambda)].
\end{equation}
So it follows the argument in Section \ref{SS:EulerComparison}  that for $\mu\in \overline{C}_{\mathbb Z}\cap X^{+}$,
\begin{equation}
\sum_{n=0}^{\dim {\mathfrak u}_{J}} (-1)^{n} \text{ch }\text{H}^{n}({\mathcal U}_{\zeta}({\mathfrak u}_{J}),L^{\zeta}(\lambda))=
\sum_{n=0}^{\dim {\mathfrak u}_{J}} (-1)^{n} \text{ch }\text{H}^{n}({\mathfrak u}_{J},L(\lambda)).
\end{equation}

\subsection{} We are now in a position to directly apply the proofs
of Theorems~\ref{T:KostantK} and ~\ref{T:KostantL} by setting $q=\zeta$, with
$\zeta$ being a primitive $l$th root of unity where $l\geq h-1$, to prove
a version of Kostant's Theorem in the root of unity case.

\begin{thm} \Label{T:Polo-Tilouine}
Let $J\subseteq \Delta$, $\mu\in X^{+}$ and $\zeta$ be a primitive $l$th
root of unity. Assume that $l\geq h-1$ and $\mu\in \overline{C}_{\mathbb Z}$.
Then as a ${\mathcal U}_{\zeta}({\mathfrak l}_{J})$-module,
$$\operatorname{H}^{n}({\mathcal U}_{\zeta}({\mathfrak u}_{J}),L^{\zeta}(\mu))
\cong \bigoplus_{\substack{w\in \WJ\\ l(w)=n}} L^{\zeta}_{J}(w\cdot \mu).$$
\end{thm}

\subsection{} We remark that one can also prove a partial converse of Kostant's
theorem as demonstrated in \cite[Theorem 5.1.1]{UGAVIGRE}.

\begin{thm} \Label{T:Kostantconverse}
Let $\zeta$ be a primitive $l$th root of unity with $l<h-1$ and $q$ a generic
parameter. Then
$$\operatorname{ch }\operatorname{H}^{\bullet}({\mathcal U}_{\zeta}({\mathfrak u}),{\mathbb C})
\neq \operatorname{ch }\operatorname{H}^{\bullet}({\mathcal U}_{q}({\mathfrak u}),{\mathbb C}).$$
\end{thm}

As is the case for comparing cohomology for complex semisimple Lie algebras versus their
modular analogs, the comparison between the cohomology for arbitrary unipotent radicals
of parabolic subalgebras and their quantum analogs is much more subtle when $l<h-1$ (cf.\  \cite[\S 6]{UGAVIGRE}).

\section{VIGRE Algebra Group at the University of Georgia}

This project was initiated during Fall Semester 2007 under the Vertical
Integration of Research and Education (VIGRE) Program sponsored
by the National Science Foundation (NSF) at the Department of Mathematics
at the University of Georgia (UGA). We would like to acknowledge the NSF VIGRE
grants DMS-0089927, and DMS-0738586 for partial financial support of this project.
The VIGRE Algebra Group consists of 3 faculty members, 1 postdoctoral fellow,
and 6 graduate students. The group is led by Brian D. Boe, Leonard Chastkofsky and Daniel K. Nakano.
The email addresses of the members of the group are given below.
\vskip .15cm
\begin{tabbing}
\hspace*{\parindent}\=Leonard Chastkofsky\quad\= \kill
Faculty: \\[2pt]
\>Brian D. Boe    \> brian@math.uga.edu \\
\>Leonard Chastkofsky \> lenny@math.uga.edu \\
\>Daniel K. Nakano    \> nakano@math.uga.edu \\[3pt]
\\
Postdoctoral Fellow: \\[2pt]
\>Benjamin Jones \> bjones@math.uga.edu \\[3pt]
Graduate Students: \\[2pt]
\>Irfan Bagci \> bagci@math.uga.edu                    \\
\>Benjamin Connell \>  bconnell@math.uga.edu                   \\
\>Wenjing Li \>  wli@math.uga.edu                       \\
\>Kenyon J. Platt    \> platt@math.uga.edu  \\
\>Jae-Ho Shin \> jshin@math.uga.edu \\
\>Caroline B. Wright \> cwright@math.uga.edu \\
\end{tabbing}

\let\section=\oldsection
\bibliographystyle{amsmath}
\bibliography{vigre5}
\end{document}